\documentclass[10pt]{article}
\usepackage[utf8]{inputenc}
\usepackage[margin=1.2in]{geometry}
\usepackage[colorlinks]{hyperref}
\usepackage{microtype}
\usepackage{mathtools, amsthm, thmtools, amsfonts, amssymb,amsmath}
\usepackage{thm-restate}
\usepackage{bm}

\usepackage{color}
\usepackage{tikz}
\usetikzlibrary{decorations.pathreplacing}

\newtheorem{thm}{Theorem}[section]
\newtheorem{theorem}[thm]{Theorem}
\newtheorem{main theorem}[thm]{Main Theorem}
\newtheorem{corollary}[thm]{Corollary}

\newtheorem{lemma}[thm]{Lemma}

\newtheorem{conjecture}[thm]{Conjecture}

\theoremstyle{definition}

\newcommand{\vb}{\mathbf{v}} 

\DeclareMathOperator{\cl}{cl}
\DeclareMathOperator{\IT}{IT}

\title{Reconfiguration of Independent Transversals}

\author{
Pjotr Buys\thanks{Korteweg--de Vries Institute for Mathematics, University of Amsterdam, Netherlands. Email: \protect\href{mailto:pjotr.buys@gmail.com}{\protect\nolinkurl{pjotr.buys@gmail.com}}, \protect\href{mailto:r.kang@uva.nl}{\protect\nolinkurl{r.kang@uva.nl}}.}
\and
Ross J. Kang\footnotemark[1]
\and
Kenta Ozeki\thanks{Faculty of Environment and Information Sciences, Yokohama National University, Japan. 
Email: \protect\href{mailto:ozeki-kenta-xr@ynu.ac.jp}{\protect\nolinkurl{ozeki-kenta-xr@ynu.ac.jp}}.}
}

\begin{document}
\maketitle

\begin{abstract}
Given integers $\Delta\ge 2$ and $t\ge 2\Delta$, suppose there is a graph of maximum degree $\Delta$ and a partition of its vertices into blocks of size at least $t$. By a seminal result of Haxell, there must be some independent set of the graph that is transversal to the blocks, a so-called independent transversal. We show that, if moreover $t\ge2\Delta+1$, then every independent transversal can be transformed within the space of independent transversals to any other through a sequence of one-vertex modifications, showing connectivity of the so-called reconfigurability graph of independent transversals.

This is sharp in that for $t=2\Delta$ (and $\Delta\ge 2$) the connectivity conclusion can fail. In this case we show furthermore that in an essential sense it can only fail for the disjoint union of copies of the complete bipartite graph $K_{\Delta,\Delta}$. This constitutes a qualitative strengthening of Haxell's theorem.
\end{abstract}
\section{Introduction}

A common refrain in combinatorics is that once you have one you have many, that is to say, combinatorial objects of a desired structure. The phenomenon, going back at least to unpublished work in 1941 of Rademacher (see~\cite{Erd55}), is referred to as {\em supersaturation} and may be observed in many, especially extremal, combinatorial contexts. We pursue a similar line that, rather than multitude, asks instead for topological structure in the space of desired combinatorial objects. Our study sits in a broader programme called {\em combinatorial reconfiguration}, see~\cite{Heu13survey}, an area that was originally motivated by various Markov chains arising in theoretical computer science. Despite there being a conceptual link between the two areas through discrete probability, there has been surprisingly little exploration of reconfiguration in the context of extremal combinatorics. Here we take a marked step in this direction.

The background for the result we build upon began in a dialogue between Erd\H{o}s and Woodall at a problem session of the 3rd British Combinatorial Conference at Merton College, Oxford in 1972 (see~\cite{Erd72,BES75}). There arose the following problem.
Given $\Delta\ge 2$, what is the least $t=t(\Delta)$ such that the following holds: for any graph of maximum degree $\Delta$ and any partition of its vertices into blocks each of order at least $t$, there is guaranteed to exist some independent set that is transversal to the vertex-partition?
In a notable application of topological methods, solving a conjecture of Bollob\'as, Erd\H{o}s and Szemer\'edi~\cite{BES75}, Haxell~\cite{Haxell1995,Haxell2001} resolved this problem by showing that a choice of $t=2\Delta$ suffices.
It is interesting to note, moreover, that for every $\Delta\ge 2$ the choice $t=2\Delta-1$ is {\em not} sufficient, as shown by Szab\'o and Tardos~\cite{SzTa06}.
As such, Haxell's result is {\em exactly} extremal, marking out its importance in (extremal) combinatorics.

Given this sharpness, it might be natural to ask for {\em more} in a sense alluded to in the first paragraph. What conditions guarantee connectedness in the space of independent transversals of a given vertex-partitioned graph? 
We give a sharp answer to this question that is analogous to Haxell's theorem.
We show that, if $t\ge2\Delta+1$, then every independent transversal (of some graph of maximum degree $\Delta$ vertex-partitioned into blocks of size at least $t$) can be modified one vertex at a time  to obtain any other, all the while remaining within the space of independent transversals. This result is also {\em exactly} extremal in that there are examples of vertex-partitioned $\Delta$-regular graphs that show how the assertion fails when $t=2\Delta$.
Below, we state a more precise result which implies the above statements and also gives a partial characterisation of  connectivity in the space of independent transversals at the threshold $t=2\Delta$.

Before that let us present our findings again in a different, perhaps more evocative way. 
Let $V(H)$ be a collection of particles and suppose each particle can take one of $t$ spins chosen from $[t]=\{1,\dots,t\}$. Suppose that the admissibility of a spin configuration on $V(H)$ is governed (on the basis of hard-core interactions) by some graph $G$ on the vertex set $V(G)=V(H)\times[t]$: 
we have an edge of $G$ between $(u,i)$ and $(v,j)$ whenever $u\ne v$ and spin $i$ on particle $u$ is incompatible with spin $j$ on particle $v$.
Equivalently, every admissible configuration is an independent set of $G$ that takes each particle of $V(H)$ in a pair exactly once. 
Consider the following Markov chain Monte Carlo algorithm which runs over the state space consisting of all admissible configurations on $V(H)$.
\begin{quote}
Initialise $X_0\subseteq V(G)$ as an arbitrary admissible configuration on $V(H)$.

Perform the following for $s=0,1,\dots$
\begin{quote} 
Choose some vertex $(u,i)$ of $G$ uniformly at random.

If $(u,i)$ is adjacent in $G$ to some vertex of $X_s$, then go directly to iteration $s+1$.

Else, let $X_{s+1}$ be the admissible configuration formed from $X_s$ by including $(u,i)$ (and so excluding the vertex of $X_s$ that contains the particle $u$).
\end{quote}
\end{quote}
While Haxell's theorem guarantees the state space of this Markov chain to be nonempty (so that $X_0$ is well-defined) provided $t\ge 2\Delta(G)$, where $\Delta(G)$ denotes the maximum degree of $G$, our main result guarantees the Markov chain to be ergodic provided $t\ge 2\Delta(G)+1$.

We remark that Glauber dynamics on the proper $t$-colourings of some graph $H$ may be formulated in a similar way as above. The ergodicity of the corresponding Markov chain provided $t\ge \Delta(H)+2$ is an exercise, as is non-emptiness of its state space when $t\ge \Delta(H)+1$; see~\cite{Jer95} for more context.

We next introduce some notation in order to state our main result more carefully, as well as set up terminology for the proof.
Let $G$ be a graph and let $\mathcal{U}$ be a partition of the vertex set $V(G)$. An independent transversal of the pair $(G,\mathcal{U})$ is an independent set $S$ of $G$ such that $|S \cap U| = 1$ for all $U \in \mathcal{U}$. The partition $\mathcal{U}$ is said to be \emph{$t$-thick} if the cardinality of each block $U \in \mathcal{U}$ satisfies $|U| \geq t$. Here is a restatement of Haxell's theorem in this notation.

\begin{theorem}[\cite{Haxell1995,Haxell2001}]
    \label{thm: Haxell}
    Let $G$ be a graph of maximum degree $\Delta$ and let $\mathcal{U}$ be a $(2\Delta)$-thick partition of $V(G)$. Then there exists an independent transversal of $(G,\mathcal{U})$.
\end{theorem}

We denote the set of independent transversals of a pair $\mathcal{G} = (G,\mathcal{U})$ by $\IT(\mathcal{G})$. We define an adjacency relation on $\IT(\mathcal{G})$ by stipulating that $S \sim T$ for $S,T \in \IT(\mathcal{G})$ if and only if $S \cup T$ is an independent set of size $|\mathcal{U}|+1$. In other words, $S$ and $T$ are adjacent if and only if they agree on all but one block $U \in \mathcal{U}$ and on $U$ their respective vertices are not adjacent in $G$ (this seemingly extraneous condition will be important to the argument). We call the graph on vertex set $\IT(\mathcal{G})$ defined by this adjacency relation the \emph{reconfigurability graph} of $\mathcal{G}$. We say that $S$ is \emph{reconfigurable} to $T$ if $S$ and $T$ lie in the same connected component of the reconfigurability graph.

Given two graph-partition pairs $\mathcal{G}_1 = (G_1, \mathcal{U}_1)$ and $\mathcal{G}_2 = (G_2, \mathcal{U}_2)$ we define their union as $\mathcal{G}_1 \cup \mathcal{G}_2 = (G_1 \sqcup G_2, \mathcal{U}_1 \cup \mathcal{U}_2)$, where $\sqcup$ denotes disjoint graph union. Observe that there is a natural bijection between $\IT(\mathcal{G}_1 \cup \mathcal{G}_2)$ and $\IT(\mathcal{G}_1) \times \IT(\mathcal{G}_2)$ and that $(S_1, S_2)$ is reconfigurable to $(T_1, T_2)$ within $\mathcal{G}_1 \cup \mathcal{G}_2$ if and only if $S_1$ is reconfigurable to $T_1$ within $\mathcal{G}_1$ and $S_2$ is reconfigurable to $T_2$ within $\mathcal{G}_2$. Therefore, when studying reconfigurablity, we will often assume that our graph-partition pairs $\mathcal{G}$ cannot be written as $\mathcal{G}_1 \cup \mathcal{G}_2$ for nontrivial $\mathcal{G}_1,\mathcal{G}_2$. In such cases, we refer to $\mathcal{G}$ as \emph{irreducible}.

We are now prepared to fully state our main theorem.

\begin{restatable}{theorem}{mainTheorem}
    \label{thm: main}
Let $G$ be a graph of maximum degree $\Delta$ and let $\mathcal{U}$ be a $(2\Delta)$-thick partition of $V(G)$ for which $\mathcal{G} = (G,\mathcal{U})$ is irreducible. If $G$ is not a disjoint union of $|\mathcal{U}|$ copies of $K_{\Delta,\Delta}$, then the reconfigurability graph of $\mathcal{G}$ is connected.
\end{restatable}

Note that if $\mathcal{U}$ is a $(2\Delta+1)$-thick partition of $V(G)$, then the hypothesis is trivially satisfied, immediately yielding one of our main claims.
By contrast, clearly there are $(2\Delta)$-thick examples where the conclusion fails when the hypothesis fails. A simple example is a single copy of $K_{\Delta,\Delta}$ with a single-block partition, but there are more intricate examples, and it would be an interesting problem to find a characterisation.

As an aside, we find it curious to juxtapose \autoref{thm: main} and a result of Aharoni, Holzman, Howard, Spr\"ussel~\cite{Aharoni2015}:
if $G$ is a graph of maximum degree $\Delta$ and $\mathcal{U}$ is a $(2\Delta-1)$-thick partition of $V(G)$, then if $G$ contains fewer than $2\Delta-1$ disjoint copies of $K_{\Delta,\Delta}$, then there exists an independent transversal of $(G,\mathcal{U})$.
The proof of \autoref{thm: main} builds conceptually upon one of the (combinatorial) proofs of \autoref{thm: Haxell} (see~\cite{Haxell2020}), even though it needs much more technical care. To aid the comparison, we give below a proof of \autoref{thm: Haxell} in our setup/notation as a warmup. (As a bonus, this warmup ensures that with respect to \autoref{thm: main} our paper is logically self-contained.)

At the end of the paper, we offer a few thoughts on further research directions leading on directly from \autoref{thm: main}.

\section{Proof of the main theorem}

In this section we will refer to a pair $\mathcal{G} = (G,\mathcal{U})$, where $\mathcal{U}$ is a $(2\Delta(G))$-thick partition of $V(G)$, as an \emph{instance}. For the rest of this section we will, given an instance, write $\Delta$ instead of $\Delta(G)$, which, to avoid trivialities, we will assume to be strictly positive.

We refer to the elements of $\mathcal{U}$ as \emph{blocks}. For $\mathcal{R} \subseteq \mathcal{U}$ we define $V_\mathcal{R} = \bigcup_{U \in \mathcal{R}} U$, we let $G_{\mathcal{R}}$ denote the induced subgraph $G[V_\mathcal{R}]$, and we let $\mathcal{G}_\mathcal{R}$ denote the instance $(G_\mathcal{R},\mathcal{R})$. For a subset of vertices $X \subseteq V(G)$ we let $\mathcal{U}_X$ denote the set of blocks that intersect $X$, i.e. $\mathcal{U}_X = \{U \in \mathcal{U}: U \cap X \neq  \emptyset\}$, and we define the \emph{closure} of $X$ as $\cl(X) = V_{\mathcal{U}_X}$. Finally we let $X_\mathcal{R} = X \cap V_\mathcal{R}$ and for a single block $U \in \mathcal{U}$ we write $X_U$ to mean $X_{\{U\}}$.

Given two independent transversals $S,T \in \IT(\mathcal{G})$ we say that $S$ and $T$ \emph{agree} on a block $U\in \mathcal{U}$ if $S_U = T_U$ and we say that $S$ and $T$ agree on a subset $\mathcal{R} \subseteq \mathcal{U}$ if $S$ and $T$ agree on every block in $\mathcal{R}$. Given a vertex $v \in V(G)$ for which $\{v\} \cup S$ is an independent set we let $S \oplus v$ denote the independent transversal $(S \setminus S_U) \cup \{v\}$, where $U$ is the block containing $v$. Note that $S$ and $S \oplus v$ are adjacent (or equal) independent transversals in the reconfigurability graph.

We proceed by proving an extremal version of \autoref{thm: main}.

\begin{lemma}
    \label{lem: disagreement_all_classes}
    Let $\mathcal{G} = (G,\mathcal{U})$ be an instance and suppose $S,T \in \IT(\mathcal{G})$ cannot be reconfigured to agree on any block. Then $G$ is isomorphic to the disjoint union of $|\mathcal{U}|$ copies of $K_{\Delta,\Delta}$.
\end{lemma}

\begin{proof}{}
    Any $v \in V(G)$ must be adjacent to at least one vertex in $S \cup T$ since otherwise $S$ and $T$ could be reconfigured to $S\oplus v$ and $T \oplus v$ respectively. We therefore find 
    \[
        2\Delta |\mathcal{U}|\leq |V(G)|\leq \sum_{v \in V(G)}|N_G(v) \cap (S \cup T)| = \sum_{u \in S \cup T}|N_G(u)| \leq \Delta \cdot |S \cup T| = 2\Delta |\mathcal{U}|.
    \]
    We find that every inequality must be an equality from which we can conclude that $\{N_G(u): u \in S \cup T\}$ partitions $V(G)$ into neighbourhoods of size $\Delta$. In particular $S \cup T$ must induce a matching between $S$ and $T$, which can be lifted to a permutation $\sigma: \mathcal{U} \to \mathcal{U}$ sending $U$ to the block containing the matched vertex in $T$ of the unique vertex in $S \cap U$.
    
    For $U \in \mathcal{U}$ let $R_U = N_G(T) \cap U$ and $B_U = N_G(S) \cap U$. We claim that any $s \in R_U$ is adjacent to any $t \in B_{\sigma(U)}$. Indeed, observe that $S \oplus s$ and $T \oplus t$ are again independent transversals, and together, by the same reasoning as above, they induce a matching in $G$ which lifts to a permutation $\sigma'$ of $\mathcal{U}$. Because $\sigma$ and $\sigma'$ are equal on $\mathcal{U}\setminus \{U\}$ they must in fact be equal on $\mathcal{U}$ and thus $s$ is adjacent to $t$.

    We can conclude that $|R_U|\leq \Delta$ and $|B_U| \leq \Delta$ for all $U \in \mathcal{U}$. On the other hand we have that $2\Delta \leq |U| = |R_U| + |B_U|$. Therefore $|R_U| = |B_U| = \Delta$ for all $U \in \mathcal{U}$. This shows that $G$ is isomorphic to the disjoint union of $|\mathcal{U}|$ copies of $K_{\Delta,\Delta}$, namely those induced by the parts $(R_U,B_{\sigma(U)})_{U \in \mathcal{U}}$.
\end{proof}

\subsection{Transforming instances}
\label{sec: transforming_instances}
The proof of \autoref{thm: main} that we give is by induction on $|\mathcal{U}|$. The induction step involves transforming an instance $\mathcal{G}$ into a smaller instance $\mathcal{G}'$. We describe this transformation in this section.

Let $\mathcal{G} = (G,\mathcal{U})$ be an instance and suppose that $G$ has a component $K$ isomorphic to $K_{\Delta,\Delta}$ with bipartition $(A,B)$. Suppose furthermore that there is a block $U \in \mathcal{U}$ such that $A \subseteq U$. Let $\mathcal{B} = \mathcal{U}_{B} \setminus \{U\}$ and let $\mathcal{R} = \mathcal{U} \setminus (\{U\} \cup \mathcal{B})$. Observe that 
\[
    \bigcup_{W \in \mathcal{B}} |W \cap B| = \Delta - |U \cap B| = 2\Delta - |U \cap (A \cup B)| \leq |U \setminus (A \cup B)|,
\]
and thus we can partition $U \setminus (A \cup B)$ into sets $\{Y(W): W \in \mathcal{B}\}$ satisfying $|B \cap W| \leq |Y(W)|$. For $W \in \mathcal{B}$ we define $W' = (W \setminus B) \cup Y(W)$, we let $\mathcal{B}' = \{W':W\in\mathcal{B}\}$, and we let $\mathcal{G}' = (G - K, \mathcal{B}' \cup \mathcal{R})$. Observe that $\mathcal{G}'$ is an instance with one fewer block than $\mathcal{G}$. We say that $\mathcal{G}'$ \emph{is obtained by swapping $B$ and $U \setminus A$}. 

We describe a map from the independent transversals of $\mathcal{G}'$ to sets of independent transversals of $\mathcal{G}$, i.e. a map $M: \IT(\mathcal{G}') \to \mathcal{P}(\IT(\mathcal{G}))$. Let $S$ be an independent transversal of $\mathcal{G}'$ and partition $\mathcal{B}'$ into $\mathcal{V}_S'\cup \mathcal{Y}_S'$ where $\mathcal{V}_S' = \{W' \in \mathcal{B}': S \cap W' \subseteq W\}$ and $\mathcal{Y}_S' = \{W' \in \mathcal{B}':S \cap W' \subseteq Y(W)\}$. If $\mathcal{Y}_S'$ is empty we let \[M(S) = \{\{a\}\cup S: a \in A\}.\] Otherwise we let
\[
    M(S) = \{\{y\}\cup X \cup S_{\mathcal{V}'_S \cup \mathcal{R}}: y \in S_{\mathcal{Y}_S'} \text{ and $X$ a transversal of }\{B \cap W: W' \in \mathcal{Y}_S'\}\}.
\]
Observe that indeed every $T\in M(S)$ is a transversal of $\mathcal{U}$ and $\bigcup_{T \in M(S)} T$ is an independent set of $G$. It follows that $M$ is a well-defined map from $\IT(\mathcal{G}')$ to $\mathcal{P}(\IT(\mathcal{G}))$ and that any pair in $M(S)$ is reconfigurable to each other. See \autoref{fig: swapping} for an example.
\begin{lemma}
    \label{lem: reconfig_transform}
Let $S_1,S_2 \in \IT(\mathcal{G}')$ be adjacent (within $\mathcal{G}'$) independent transversals. Then any independent transversal in $M(S_1)$ is reconfigurable (within $\mathcal{G}$) to any independent transversal in $M(S_2)$.
\end{lemma}
\begin{proof}{}
    It is sufficient to show that there are independent transversals $T_1\in M(S_1)$ and $T_2 \in M(S_2)$ such that $T_1$ is reconfigurable to $T_2$ within $\mathcal{G}$.

    Let $v, w \in V(G-K)$ such that $S_2 = S_1 \setminus \{v\} \cup \{w\}$ and let $T_1 \in M(S_1)$. If both $v,w$ are elements of a block in $\mathcal{R}$ or if $v,w$ are either both in $Y(W)$ or both in $W' \setminus Y(W)$ for some $W \in \mathcal{B}$ then $T_1 \oplus w \in M(S_2)$. We will thus assume that there is a block $W \in \mathcal{B}$ such that $v \in W' \setminus Y(W)$ and $w \in Y(W)$. We take $b \in B\cap W$.

    If $\mathcal{Y}_{S_1}'$ is empty then $T_1 \to T_1\oplus w \to (T_1\oplus w)\oplus b$ is a valid reconfiguration from an independent transversal in $M(S_1)$ to an independent transversal in $M(S_2)$. If $\mathcal{Y}_{S_1}'$ is not empty then $T_1 \to T_1 \oplus b$ is a valid reconfiguration from an independent transversal in $M(S_1)$ to an independent transversal in $M(S_2)$.
\end{proof}

\begin{figure}[ht!]
    \centering
    \begin{tikzpicture}[scale=0.95]
        % Rectangles belonging to R
        \draw (0,0) rectangle (1.25,4);
        \draw (-1.5,0) rectangle (-0.25,4);
        \draw (-3,0) rectangle (-1.75,4);
        \draw (-4.5,0) rectangle (-3.25,4);
        \draw (-6,0) rectangle (-4.75,4);
        % Bounding box
        \draw (-6.75,-1) rectangle (2,5);
        \node at (-2.375,5.5) {\large$\mathcal{G}$};

        % Rectangles belonging to R'
        \draw (3,0) rectangle (4.25,4);
        \draw (4.5,0) rectangle (5.75,4);
        \draw (6,0) rectangle (7.25,4);
        \draw (7.5,0) rectangle (8.75,4);
        %Bounding Box
        \draw (2.25,-1) rectangle (9.5,5);
        \node at (5.625,5.5) {\large$\mathcal{G}'$};

        \fill (1.1,0.25) circle (0.075);
        \node at (1.4,0.25) {\small$y_3$};
        
        %First rectangle
        \draw (0,2) -- (1.25,2);
        \draw (0,0.5) -- (1.25, 0.5);
        \draw (0,1) -- (1.25, 1);
        \draw (0,1.5) -- (1.25, 1.5);
        \node at (0.5625,0.25) {\small$Y(W_3)$};
        \node at (0.5625,0.75) {\small$Y(W_2)$};
        \node at (0.5625,1.25) {\small$Y(W_1)$};
        \node at (0.5625,1.75) {\small$B$};
        \node at (0.5625,3) {\small$A$};
        \node at (0.5625,-0.35) {$U$};
        
        % Decorations on the B blocks
        \draw (-1.5,0.5) -- (-0.25, 0.5);
        \node at (-1,0.25) {\small$B$};
        \fill (-0.5,0.25) circle (0.075);
        \node at (-0.875,-0.35) {$W_3$};

        \draw (-3,0.5) -- (-1.75, 0.5);
        \node at (-2.5,0.25) {\small$B$};
        \fill (-2,0.25) circle (0.075);
        \node at (-2.375,-0.35) {$W_2$};
        
        \draw (-4.5,0.5) -- (-3.25, 0.5);
        \node at (-3.875,0.25) {\small$B$};
        \node at (-3.875,-0.35) {$W_1$};
        \fill (-3.875,1.5) circle (0.075);
        \node at (-3.575,1.5) {\small$v_1$};

        \fill (-5.375,1.25) circle (0.075);
        \node at (-5.075,1.25) {\small$v_0$};

        \draw[decorate,decoration={brace,amplitude=5pt,raise=2pt}] (-4.5,4) -- node[above=6pt] {$\mathcal{B}$} (-0.25,4);

        \draw[decorate,decoration={brace,amplitude=5pt,raise=2pt}] (-6,4) -- node[above=6pt] {$\mathcal{R}$} (-4.75,4);

        % Decorations on the B' blocks
        \draw (4.5,0.5) rectangle (5.75,0.5);
        \node at (5.125,0.25) {\small$Y(W_1)$};
        \node at (5.25,-0.35) {$W_1'$};
        \fill (5.125,1.5) circle (0.075);
        \node at (5.425,1.5) {\small$v_1$};

        \fill (3.625,1.25) circle (0.075);
        \node at (3.925,1.25) {\small$v_0$};
        
        \draw (6,0.5) rectangle (7.25,0.5);
        \node at (6.5,0.25) {\small$Y(W_2)$};
        \fill (7.1,0.25) circle (0.075);
        \node at (7.4,0.25) {\small$y_2$};
        \node at (6.75,-0.35) {$W_2'$};
        
        \draw (7.5,0.5) rectangle (8.75,0.5);
        \node at (8,0.25) {\small$Y(W_3)$};
        \fill (8.6,0.25) circle (0.075);
        \node at (8.9,0.25) {\small$y_3$};
        \node at (8.25,-0.35) {$W_3'$};

        \draw[decorate,decoration={brace,amplitude=5pt,raise=2pt}] (4.5,4) -- node[above=6pt] {$\mathcal{B}'$} (8.75,4);

        \draw[decorate,decoration={brace,amplitude=5pt,raise=2pt}] (3,4) -- node[above=6pt] {$\mathcal{R}$} (4.25,4);

        % The u transversal in G'
        \fill (3.625,3.5) circle (0.075);
        \fill[white] (3.625,3.5) circle (0.05);
        \node at (3.925,3.5) {\small$u_0$};
        
        \fill (5.125,3.25) circle (0.075);
        \fill[white] (5.125,3.25) circle (0.05);
        \node at (5.425,3.25) {\small$u_1$};

        \fill (6.625,3.4) circle (0.075);
        \fill[white] (6.625,3.4) circle (0.05);
        \node at (6.925,3.4) {\small$u_2$};

        \fill (8.125,3.1) circle (0.075);
        \fill[white] (8.125,3.1) circle (0.05);
        \node at (8.425,3.1) {\small$u_3$};

        % The u transversal in G
        \fill (-5.375,3.5) circle (0.075);
        \fill[white] (-5.375,3.5) circle (0.05);
        \node at (-5.075,3.5) {\small$u_0$};

        \fill (-3.875,3.25) circle (0.075);
        \fill[white] (-3.875,3.25) circle (0.05);
        \node at (-3.575,3.25) {\small$u_1$};

        \fill (-2.375,3.4) circle (0.075);
        \fill[white] (-2.375,3.4) circle (0.05);
        \node at (-2.075,3.4) {\small$u_2$};

        \fill (-0.875,3.1) circle (0.075);
        \fill[white] (-0.875,3.1) circle (0.05);
        \node at (-0.575,3.1) {\small$u_3$};

        \fill (0.625,3.4) circle (0.075);
        \fill[white] (0.625,3.4) circle (0.05);

    \end{tikzpicture}
    \caption{An example of an instance $\mathcal{G}'$ obtained by swapping $B$ and $U\setminus A$ in $\mathcal{G}$. The sets $S_1 = \{u_0,u_1,u_2,u_3\}$ and $S_2 = \{v_0,v_1,y_2,y_3\}$ are two examples of independent transversals in $\mathcal{G}'$ an element from both $M(S_1)$ and $M(S_2)$ is drawn in $\mathcal{G}$.}
    \label{fig: swapping}
\end{figure}
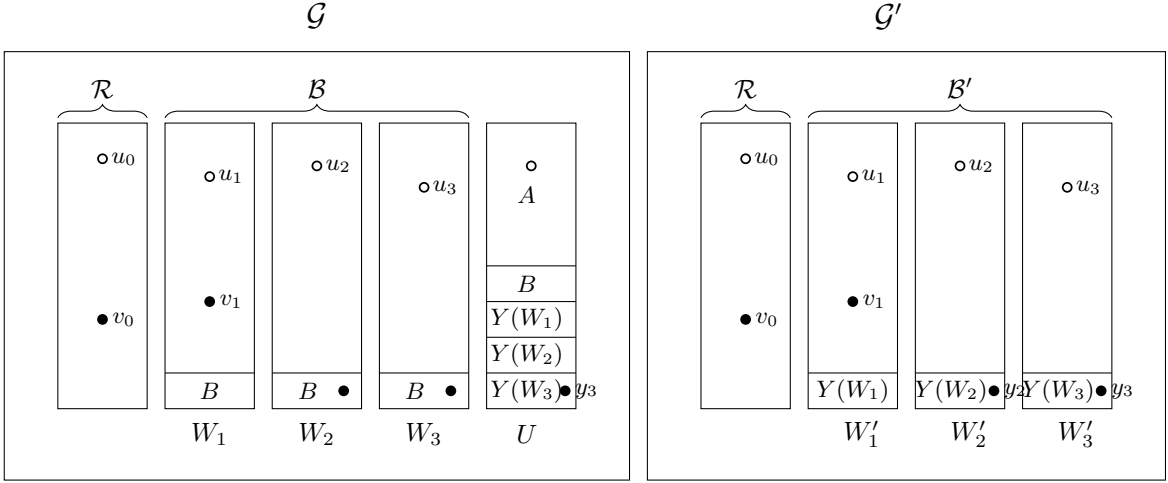

\subsection{Descent using augmenting sequences}
Let $\mathcal{G}=(G,\mathcal{U})$ be an instance, let $S$ be a partial independent transversal of $\mathcal{G}$, i.e. an independent set for which $|S_U| \leq 1$ for all $U \in \mathcal{U}$, and let $C \subseteq V_{\mathcal{U} \setminus \mathcal{U}_{S}}$. Given a sequence of vertices $\vb = (v_1, \dots, v_m)$ we define the sequence of blocks $\mathcal{B}_n$ by $\mathcal{B}_0 = \mathcal{U} \setminus \mathcal{U}_{S}$ and $\mathcal{B}_n = \mathcal{B}_{n-1} \cup \mathcal{U}_{N_S(v_n)}$ for $n=1,\dots, m$ and we define the sequence of vertex sets $C_n$ by $C_0=C$ and $C_n= C_{n-1} \cup N_S[v_n]$ for $n=1,\dots, m$. We define the degree sequence $d(\vb) = (d_1(\vb),\dots,d_m(\vb))$ with $d_k(\vb) = |N_{S}(v_k)|$ for all $k = 1, \dots, m$, and the number of degrees is denoted by $\textrm{length}(\vb)$, that is, $\textrm{length}(\vb) = m$. We say that $\vb$ is \emph{augmenting} for $(S,C)$ if for all $k = 1, \dots, m$ we have 
\begin{enumerate}
    \item $v_k \in V_{\mathcal{B}_{k-1}}$;
    \item $v_k$ is not adjacent to any vertex in $C_{k-1}$; and
    \item $d_k(\vb) > 0$ if $k\neq m$.
\end{enumerate}
The empty sequence, denoted by $\varnothing$, is considered augmenting. We call a tuple $(S,C,\vb)$ \emph{feasible} for $\mathcal{G}$ if $S$ and $C$ are as described above and $\vb$ is augmenting for $(S,C)$.

We introduce a strict partial order on the set of all feasible tuples. We say that $(S,C,\vb) < (S',C',\vb')$ if $|S| > |S'|$ or if $|S'| = |S|$ and there is an index $1 \leq k \leq \textrm{length}(\vb)$ such that $d_i(\vb) = d_i(\vb')$ for all $i < k$ and either $k > \textrm{length}(\vb')$ or $d_k(\vb) < d_k(\vb')$. This condition is equivalent to the following lexicographical inequality of infinite padded sequences
\[
    (d_1(\vb),\dots,d_m(\vb),\infty,\infty,\dots) < (d_1(\vb'),\dots,d_{m'}(\vb'),\infty,\infty,\dots).
\]
We remark that $(S,C, \mathbf{v}) < (S,C,\varnothing)$ for any nonempty sequence $\mathbf{v}$ that is augmenting for $(S,C)$.
It follows directly from the definition that, given a feasible tuple $(S,C,\vb)$, we have
\begin{equation}
    \label{eq: bound_sets}
    |C_k| \leq |C| + k + \sum_{i=1}^k d_i(\vb) 
    \quad
    \text{ and }
    \quad
    |\mathcal{B}_k| = |\mathcal{U} \setminus \mathcal{U}_S| + \sum_{i=1}^k d_i(\vb).
\end{equation}
Haxell's theorem (c.f. \autoref{thm: Haxell}) can be proved via a descent through the space of feasible tuples as is sketched in \cite{Haxell2020}. To familiarise ourselves with the technique we reproduce the argument with our notation.
\begin{proof}[Proof of Haxell's theorem]
    Let $\mathcal{G}=(G,\mathcal{U})$ be an instance and let $(S,\emptyset,\vb)$ with $\vb = (v_1, \dots, v_m)$ be a minimal feasible tuple for $\mathcal{G}$.
    We will prove that $S$ is a full independent transversal. We suppose otherwise, i.e. we suppose that $\mathcal{U}\setminus \mathcal{U}_S$ is not empty.

    We first claim $\vb$ is nonempty and that $d_m(\vb) = 0$. Otherwise 
    \[
        |V_{\mathcal{B}_m}| \geq 2\Delta \left(|\mathcal{U\setminus}\mathcal{U}_S| + \sum_{i=1}^m d_i(\vb)\right) \geq \Delta\left(2|\mathcal{U}\setminus\mathcal{U}_S| + m + \sum_{i=1}^m d_i(\vb)\right) > \Delta \cdot |C_m|.
    \]
    This implies that there is a vertex $v_{m+1}\in V_{\mathcal{B}_m}$ independent of $C_m$ and thus the sequence $\vb' = (v_1, \dots, v_m,v_{m+1})$ is augmenting for $(S,\emptyset)$ and thus $(S,\emptyset,\vb')$ is a strictly smaller feasible tuple than $(S,\emptyset,\vb)$. This contradicts minimality.

    We may thus assume that $m \geq 1$ and $d_m(\vb) = 0$. Let $k$ be the smallest index for which $v_m \in V_{\mathcal{B}_k}$. If $k = 0$ then $v_m \in V_{\mathcal{U} \setminus \mathcal{U}_S}$ and thus $S' = S \cup \{v_m\}$ is strictly larger partial independent transversal than $S$. This implies that $(S',\emptyset, \varnothing) < (S,\emptyset,\vb)$, contradicting minimality. If $k \geq 1$ let $S' = S \oplus v_m$ and $\vb' = (v_1, \dots, v_k)$. Now $\vb'$ is augmenting for $(S', \emptyset)$, $|S'| = |S|$, $d_i(\vb') = d_i(\vb)$ for $i < k$, and $d_k(\vb') = d_k(\vb)-1$. It follows that $(S',\emptyset,\vb') < (S,\emptyset,\vb)$, which again contradicts minimality.
\end{proof}

The proof of the following lemma is very similar. Its corollary will be used in the proof of \autoref{thm: main}.

\begin{lemma}
    \label{lem: precolor2}
    Let $\mathcal{G} = (G,\mathcal{U})$ be an instance with independent transversal $S$. Moreover, let $X \subseteq V(G)$ with $|X| < \Delta$. Then $S$ can be reconfigured to $S'$ satisfying $S' \cap X = \emptyset$. 
\end{lemma}
\begin{proof}
    We can reconfigure $S$ to $T$ where $(T\setminus X, T \cap X, \vb)$ is a minimal element within the following set:
     \[
        \{(T\setminus X,T \cap X, \vb): S \text{ can be reconfigured to $T$ and $\vb \cap X = \emptyset$}\}.
    \]
    We claim that $T \cap X = \emptyset$. To show this we assume otherwise and we will find that this contradicts the minimality of $(T\setminus X, T \cap X, \vb)$. We write $\vb = (v_1, \dots, v_m)$.

    Suppose $m \geq 1$ and $d_m(\vb)=0$. Then $v_m$ is independent of $T$ and thus we can reconfigure $T$ to $T'=T\oplus v_m$. If $v_m \in \cl(T \cap X)$ then $T'$ is an independent transversal satisfying $|T'\setminus X | > |T\setminus X|$, contradicting the minimality of $(T\setminus X, T \cap X, \vb)$. Otherwise there is an index $k < m$ such that $v_m \in \mathcal{U}_{N_T(v_k)}$. In this case we let $\vb' = (v_1, \dots, v_k)$ and we observe that $\vb'$ is feasible for $(T'\setminus X, T' \cap X)$ with $d_i(\vb') = d_i(\vb)$ for $i < k$ and $d_k(\vb') = d_k(\vb)-1$ and thus $(T'\setminus X,T' \cap X, \vb') < (T\setminus X,T \cap X, \vb)$.

    We can thus assume that $d_i(\vb) \geq 1$ for $i=1, \dots, m$ (where we now allow $m=0$). Then we see, using equation~{\eqref{eq: bound_sets}}, that 
    \[
        |V_{\mathcal{B}_m}| \geq 2\Delta\left(|T\cap X| + \sum_{i=1}^m d_i(\vb)\right)
        \geq \Delta\left(1 + |T \cap X| + m + \sum_{i=1}^m d_i(\vb)\right)
        > |X| + |N_G(C_m)|.
    \]
    This means that there is a vertex $u \in V_{\mathcal{B}_m}\setminus X$ independent of $C_m$, and thus $\vb' = (v_1, \dots, v_m, u)$ is an augmenting sequence for $(T\setminus X, T \cap X)$ with $(T\setminus X,T \cap X, \vb') < (T\setminus X,T \cap X, \vb)$ contradicting the minimality assumption. This concludes the proof.
\end{proof}

\begin{corollary}
    \label{cor: special reconfiguration2}
    Let $\mathcal{G}=(G,\mathcal{U})$ be an instance with $S,T \in \IT(\mathcal{G})$ and let $\mathcal{R} \subseteq \mathcal{U}$ such that $S,T$ agree on $\mathcal{R}$ and $N_G((S\cup T)\cap V_{\mathcal{U}\setminus \mathcal{R}}) \subseteq V_{\mathcal{U}\setminus \mathcal{R}}$. Suppose $x \in V_{\mathcal{U}\setminus \mathcal{R}}$ has at least one neighbour in $V_{\mathcal{U}\setminus \mathcal{R}}$. Then $(S,T)$ can be reconfigured to $(S_x,T_x)$ such that $S_x$ and $T_x$ agree on $\mathcal{R}$, both $S$ agrees with $S_x$ and $T$ agrees with $T_x$ on $\mathcal{U} \setminus \mathcal{R}$, and $x$ is independent of $(S_x \cup T_x) \cap V_\mathcal{R}$.
\end{corollary}
\begin{proof}{}
    Observe that $R := S_\mathcal{R} \cap V_\mathcal{R} = T_\mathcal{R} \cap V_\mathcal{R}$ is an independent transversal of $\mathcal{G}_\mathcal{R}$. Let $X = N_G(x) \cap V_\mathcal{R}$ and observe that $|X| < \Delta$. It follows from the previous lemma that $R$ can be reconfigured on $\mathcal{G}_\mathcal{R}$ to $R'$ with $R' \cap X = \emptyset$. Because $N_G((S\cup T)\cap V_{\mathcal{U}\setminus \mathcal{R}}) \subseteq V_{\mathcal{U}\setminus \mathcal{R}}$ this corresponds to a valid reconfiguration of $(S,T)$ to $(S_x,T_x)$.
\end{proof}

\subsection{Reconfiguring independent transversals}

In this section we will prove \autoref{thm: main}. For two sets $X,Y$ we let $X \triangle Y = (X \cup Y) \setminus (X \cap Y)$ denote the symmetric difference of $X$ and $Y$. Given an instance $\mathcal{G}=(G,\mathcal{U})$ with $S_0,T_0 \in \IT(\mathcal{G})$ we define the set 
\[
    {\bf{X}}(S_0,T_0) = \{(S,T,\mathbb{\vb}): \text{$(S_0,T_0)$ is reconfigurable to $(S,T)$ and $\vb$ is augmenting for $(S\cap T, S\triangle T)$}\}.
\]
The set ${\bf{X}}(S_0,T_0)$ is endowed with a strict partial order inherited by the strict partial order on the feasible tuples $(S\cap T, S\triangle T,\vb)$. Observe that $S_0$ can be reconfigured to $T_0$ if and only if the minimal elements of ${\bf{X}}(S_0,T_0)$ are of the form $(S',S',\varnothing)$. Given $(S,T,\mathbf{v}) \in {\bf{X}}(S_0,T_0)$ equation~{\eqref{eq: bound_sets}} simplifies to 
\begin{equation}
    \label{eq: bound_sets_special}
    |C_k| \leq |S \triangle T| + k + \sum_{i=1}^k d_i(\vb)
    \quad
    \text{ and }
    \quad
    |\mathcal{B}_k| = \frac{1}{2}|S \triangle T| + \sum_{i=1}^k d_i(\vb).
\end{equation}

\begin{lemma}
    \label{lem: 2or0 implies smaller seq2}
    Let $\mathcal{G}$ be an instance with $S_0,T_0 \in \IT(\mathcal{G})$ and let $(S,T,\vb) \in {\bf X}(S_0,T_0)$ with $\vb = (v_1, \dots, v_m)$. Suppose $m \geq 1$ and either $d_m(\vb)=0$ or $d_k(\vb) \geq 2$ for some $k$. Then $(S,T,\vb)$ is not minimal in ${\bf X}(S_0,T_0)$.
\end{lemma}

\begin{proof}{}
    Assume first that $d_m(\vb)=0$. This means that $v_m$ is not adjacent to any element of $S \cup T$ and thus $(S,T)$ can be reconfigured to $(S',T') = (S \oplus v_m, T\oplus v_m)$.

    Let $k$ be the smallest index for which $v_m \in C_k(\vb)$. If $k = 0$ then $v_m \in \cl(S \triangle T)$ and thus $S'$ and $T'$ agree on one more block and thus $(S',T',\emptyset)< (S,T,\vb)$. Otherwise, we see that $v_m \in \cl(N_{S,T}(v_k))$. Now $\vb' = (v_1, \dots, v_k)$ is an augmenting sequence for $(S',T')$ satisfying $d_i(\vb') = d_i(\vb)$ for $i < k$ and $d_k(\vb') = d_k(\vb)-1$, which implies $(S',T',\vb')< (S,T,\vb)$. This proves the statement if $d_m(\vb)=0$. 
    
    Now we assume that $d_m(\vb)\geq 1$ and there is some index $k$ for which $d_k(\vb)\geq 2$. It follows from equation~{\eqref{eq: bound_sets_special}} that 
    \begin{align*}
        \Delta |C_m| \leq \Delta \left(|S \triangle T| + m + \sum_{i=1}^m d_i(\vb)\right)
        < 2\Delta \left(\frac{1}{2}|S \triangle T| + \sum_{i=1}^m d_i(\vb)\right) \leq |V_{\mathcal{B}_m}|.
    \end{align*}
    It follows that there is a vertex $v_{m+1} \in V_{\mathcal{B}_m}$ not adjacent to any vertex in $C_m$. Therefore, letting $\vb'= (v_1, \dots, v_m,v_{m+1})$, we find $(S,T,\vb') \in {\bf X}(S_0,T_0)$ with $(S,T,\vb') < (S,T,\vb)$.
\end{proof}

\begin{lemma}
    \label{lem: nbhs partition}
    Let $\mathcal{G}=(G,\mathcal{U})$ be an instance with $S_0,T_0 \in \IT(\mathcal{G})$ and let $(S,T,\vb)$ be minimal inside ${\bf{X}}(S_0,T_0)$. Then $\{N_G(u): u \in C_m\}$ partitions $V_{\mathcal{B}_m}$, where $m = \rm{length}(\vb)$.
\end{lemma}

\begin{proof}{}
If $S = T$ then $\vb$ is the empty sequence and both $C_0$ and $V_{\mathcal{B}_0}$ are empty and thus the statement is trivially true.

Suppose now $S \neq T$ and write $\vb = (v_1,\dots, v_m)$. It follows from \autoref{lem: 2or0 implies smaller seq2} that $d_i(\vb) = 1$ for all $i$. We also observe that any vertex in $V_{\mathcal{B}_m}$ must be adjacent to some vertex in $C_m$. From equation~{\eqref{eq: bound_sets_special}} we obtain
$\Delta |C_m| \leq \Delta(|S \triangle T| + 2m) \leq |V_{\mathcal{B}_m}|$, which shows that the neighbourhoods of elements in $C_m(\vb)$ must all have size $\Delta$, cannot intersect, and must cover $V_{\mathcal{B}_m}$. 
\end{proof}

\begin{lemma}
    \label{lem: nonempty seq}
    Let $\mathcal{G}=(G,\mathcal{U})$ be an irreducible instance with $S_0,T_0 \in \IT(\mathcal{G})$ and let $(S,T,\vb)$ be minimal inside ${\bf{X}}(S_0,T_0)$. Then either $S \cap T = \emptyset$, $S = T$, or $\vb \neq \varnothing$.
\end{lemma}

\begin{proof}{}
We suppose that neither $S\cap T$ nor $S \triangle T$ is empty but that $\vb$ is empty. It follows from \autoref{lem: nbhs partition} that $\{N_G(u): u \in S \triangle T\}$ partitions $\cl(S\triangle T)$. Let $\mathcal{R} = \mathcal{U}_{S \cap T}$. Because $\mathcal{G}$ is irreducible there is a vertex $x \in V_{\mathcal{U}\setminus \mathcal{R}}$ that has a neighbour in $V_\mathcal{R}$. This implies that $x \not \in S \cup T$. The vertex $x$ must have a unique neighbour in $S \triangle T$, say in $T$. It follows from \autoref{cor: special reconfiguration2} that $(S,T)$ can be reconfigured to $(S_x,T_x)$ such that $S_x$ and $T_x$ agree on $\mathcal{R}$, both $S$ agrees with $S_x$ and $T$ agrees with $T_x$ on $\mathcal{U} \setminus \mathcal{R}$, and $(S_x \cup T_x) \cap V_\mathcal{R} \cap N_G(x) = \emptyset$. It follows that $S_x$ is independent of $x$ and thus $(S,T)$ can be reconfigured to $(S',T'):= (S_x\oplus x, T_x)$.

The transversals $S'$ and $T'$ agree on exactly the same blocks as $S$ and $T$ and thus $(S',T',\varnothing)$ is minimal in ${\bf{X}}(S_0,T_0)$. This implies that $\{N_G(u): u \in S' \triangle T'\}$ partitions $\cl(S'\triangle T') = \cl(S\triangle T)$. This contradicts the fact that $x$ has a neighbour in $V_{\mathcal{R}}$, which concludes the proof.
\end{proof}

We are now ready to prove the main theorem, which we restate for convenience. 
\mainTheorem*

\begin{proof}{}

The proof is by induction on $|\mathcal{U}|$. For $|\mathcal{U}| = 1$ the statement follows from \autoref{lem: disagreement_all_classes}. We will thus assume that $|\mathcal{U}| > 1$ and that the statement holds for instances with strictly fewer than $|\mathcal{U}|$ blocks. Let $S_0,T_0 \in \IT(\mathcal{G})$ and let $(S,T,\vb)$ be a minimal element in ${\bf X}(S_0,T_0)$ and write $\vb = (v_1, \dots, v_m)$. If $S=T$ we are done and thus we will assume that $S \neq T$.

It follows from \autoref{lem: disagreement_all_classes} that $|S \cap T|$ is not empty. By \autoref{lem: 2or0 implies smaller seq2} and \autoref{lem: nonempty seq} it follows that $\vb$ is not empty and $d_i(\vb)=1$ for $i=1, \dots, m$. By \autoref{lem: nbhs partition} $\{N_G(u): u \in C_m\}$ is a partition of $V_{\mathcal{B}_m}$. Let $x_i$ be the unique neighbour of $v_i$ in $S \cup T$. Note that $C_m$ can be written as the disjoint union of $S \triangle T$, $\{v_1, \dots, v_m\}$, and $\{x_1, \dots, x_m\}$ and the degree of any of these vertices is exactly $\Delta$.

Let $\mathcal{R} = \mathcal{U} \setminus \mathcal{B}_m$. It follows from \autoref{cor: special reconfiguration2} that, given an $x \in N_G(v_i)$, $(S,T)$ can be reconfigured to $(S_x,T_x)$ such that $S_x$ and $T_x$ both contain $x$, still agree on $\mathcal{R}$, and not only $S$ agrees with $S_x$ but also $T$ agrees with $T_x$ on every block in $\mathcal{B}_m$ except possibly the block containing $x$. 

For $i = 1, \dots, m$ let $X_i$ be the block containing $x_i$. We claim that $N_G(v_m) \subseteq X_m$. Indeed, suppose that there is some $x \in N_G(v_m)$ with $x \not\in X_m$. Then either $x \in \cl(S\triangle T)$, which would imply $|S_x \triangle T_x| < |S \triangle T|$ contradicting the minimality of $(S,T, \vb)$ within ${\bf X}(S_0,T_0)$. Or $x \in X_k$ for some $k < m$. But then $\vb' := (v_1, \dots, v_k)$ is an augmenting sequence for $(S_x\cap T_x,S_x \triangle T_x)$ with $d_k(\vb') = 0$ and thus $\vb' < \vb$. This again contradicts minimality.

Next we claim every $x \in N_G(v_m)$ has the same set of neighbours. To see this, take such an $x$ and observe that $\vb$ is an augmenting sequence for $(S_x\cap T_x,S_x \triangle T_x)$ with $d_{i}(\vb) = 1$ for all $i$. By \autoref{lem: nbhs partition} 
\[
    \{N_G(u): u \in (S \triangle T)\cup \{v_1, \dots, v_m\} \cup \{x_1, \dots, x_{m-1},x\}\} 
\]
is a partition of $V_{\mathcal{U} \setminus \mathcal{R}}$. This shows that $N_G(x)$ does not depend on $x$. We conclude that the edge $x_mv_m$ lies in a component $K$ of $G$ isomorphic to $K_{\Delta,\Delta}$. We refer to the sides containing $x_m$ and $v_m$ as $A$ and $B$ respectively and we recall that we showed that $A \subseteq X_m$.

Let $Y = X_m \setminus (A \cup B)$ and let $\mathcal{G}' = (G - K, \mathcal{U}')$ be an instance obtained from swapping $B$ and $X_m \setminus A$; see Section~\ref{sec: transforming_instances}. Recall that there is a bijection $\mathcal{U} \setminus \{X_m\} \to \mathcal{U}'$ by sending $W$ to $W'$ with $W=W'$ if $W \cap B = \emptyset$ and $W' = (W \setminus B) \cup Y(W)$ otherwise. The independent transversals $(S,T)$ of $\mathcal{G}$ restrict to independent transversals $(S_{\mathcal{G}'},T_{\mathcal{G}'})$ of $\mathcal{G}'$. We claim that any irreducible component of $\mathcal{G}'$ contains a block on which $S_{\mathcal{G}'}$ and $T_{\mathcal{G}'}$ disagree. Observe that for $R \in \mathcal{R}$ we have $R=R'$ and thus, since $\mathcal{G}$ is irreducible, any irreducible component of $\mathcal{G}'$ contains a block in $\mathcal{U}'\setminus \mathcal{R}$. Let $\mathcal{G}'_C$ be such a component and assume it does not contain any block on which $S_{\mathcal{G}'}$ and $T_{\mathcal{G}'}$ disagree. Since it must contain a block in $\mathcal{U}'\setminus \mathcal{R}$ it must contain a block $X_i'$ for some $i \in [m-1]$ and we may assume that $i$ is chosen minimally. By definition either $v_i \in X_j$ for $j < i$ or $v_i \in \cl(S\triangle T)$ within $\mathcal{G}$, and since the edge $v_ix_i$ is not an edge of $K$ this would imply that either $v_i \in X_j'$ or $v_i \in \cl(S_{\mathcal{G}'} \triangle T_{\mathcal{G}'})$ within $\mathcal{G}'$. Since $i$ was chosen minimally this contradicts the fact that $\mathcal{G}'_C$ is a component that does not contain any block on which $S_{\mathcal{G}'}$ and $T_{\mathcal{G}'}$ disagree, which proves the claim.

Because $G$ is not a disjoint union of $|\mathcal{U}|$ copies of $K_{\Delta,\Delta}$ there must be a component $\mathcal{G}_C' = (G_C, \mathcal{U}_C')$ of $\mathcal{G}'$ for which $G_C$ is not the disjoint union of $|\mathcal{U}_C'|$ copies of $K_{\Delta,\Delta}$. By induction we can reconfigure $T_{\mathcal{G}'}$ to $S_{\mathcal{G}'}$ on this component. This corresponds to a reconfiguration from $T_{\mathcal{G}'}$ to $\tilde{T}_{\mathcal{G}'}$ on $\mathcal{G}'$ for which $|T_{\mathcal{G}'} \cap S_{\mathcal{G}'}| < |\tilde{T}_{\mathcal{G}'} \cap S_{\mathcal{G}'}|$. The independent transversal $\tilde{T}_{\mathcal{G}'}$ agrees with either $S_{\mathcal{G}'}$ or $T_{\mathcal{G}'}$ on every block of $\mathcal{G}'$ and thus it does not intersect $Y$. By \autoref{lem: reconfig_transform} it follows that $T$ can be reconfigured to $\tilde{T}$ within $\mathcal{G}$ that agrees with $S$ on $X_m$ and on every block in $U \in \mathcal{U}\setminus \{X_m\}$ for which $S_{\mathcal{G}'}$ and $\tilde{T}_{\mathcal{G}'}$ agree on $U'$. It follows that 
\[
    |S \cap \tilde{T}|  = |S_{\mathcal{G}'} \cap \tilde{T}_{\mathcal{G}'}| + 1
                        > |S_{\mathcal{G}'} \cap T_{\mathcal{G}'}| + 1
                        = |S \cap T|.
\]
This contradicts the assumption that the pair $S,T$ was reconfigured to agree on a maximal number of blocks.
\end{proof}

\section{Concluding remarks}
Our contribution in this work was to conceive of and establish \autoref{thm: main}, which one may view as a qualitative strengthening, or a reconfiguration-analogue, of Haxell's theorem. From this result there open up numerous intriguing lines for further research, which we briefly outline.
\begin{itemize}
\item 
Since the class of graphs formed from the disjoint union of copies of $K_{\Delta,\Delta}$ seems eminently tractable, \autoref{thm: main} hints temptingly at the feasibility of the following problem. What is an exact characterisation (for part size $t=2\Delta$) of those irreducible vertex-partitioned $\Delta$-regular graphs for which the reconfigurability graph of independent transversals is disconnected?
\item
What are the best upper bounds, as a function of $|\mathcal{U}|$ and $\Delta$ on the diameter of the reconfigurability graph of $\mathcal{G}$ in the conclusion of \autoref{thm: main}? Similarly, what are the best diameter upper bounds for the reconfigurability graph under the assumption that $|\mathcal{U}|$ is $(2\Delta+1)$-thick?
\item
We presented a Markov chain, which by \autoref{thm: main} settles into an equilibrium distribution over all admissible configurations. When does it mix rapidly, that is, polynomially fast as a function of $|V(H)|$?
Could it be shown provided $t$ is large enough in comparison to $\Delta$, say, some constant factor $C\ge2$ larger? Put another way, is there some independent transversals analogue of Jerrum's result on Glauber dynamics for proper colourings of bounded degree graphs~\cite{Jer95}?
\end{itemize}

We referred in passing to the fact that independent transversals nontrivially generalise proper colourings of graphs. There are interesting notions intermediate to proper colourings and independent transversals, where one could also pursue a study analogous to ours. In fact, \autoref{thm: main} validates all such studies. We highlight two such (related) notions, especially as they may demand techniques different from those we used here.

\begin{itemize}
\item
Let $H$ be a graph. Suppose that to each vertex $u\in V(H)$ we associate a list $L(u)\subseteq \mathbb{Z}^+$ of $t$ positive integers. Let us furthermore restrict ourselves to such $L$ satisfying for each $u$ and each $c \in L(u)$ that $|\{v\in N_H(u) \, \mid \, L(v) \ni c\}| \le \Delta$, that is, the maximum {\em colour-degree} is bounded by $\Delta$. A proper $L$-colouring of $H$ is a function $f: V(H) \to \mathbb{Z}^+$ that makes a choice $f(u)\in L(u)$ for each $u$ from its own list so that $f(u)\ne f(v)$ for any edge $uv\in V(H)$. 
Reed and Sudakov~\cite{ReSu02} showed that some $t$ satisfying $t=\Delta+o(\Delta)$ as  $\Delta\to\infty$ suffices to always guarantee a proper $L$-colouring under the above setup, while Bohman and Holzman~\cite{BoHo02} showed that $t=\Delta+1$ does {\em not} suffice. What is the best choice of $t=t(\Delta)$ that guarantees connectedness in the space of proper $L$-colourings (changing the choice for one vertex of $H$ at a time) under this setup?
\item
Given a graph $G$, we say that a partition $\mathcal{U}$ of $V(G)$ has maximum {\em local degree} $s$ if the maximum degree of any subgraph induced by two of the parts is at most $s$. Loh and Sudakov~\cite{LoSu07} showed that provided $s=s(\Delta)$ satisfies $s=o(\Delta)$ as $\Delta\to\infty$, there is some choice of $t=t(\Delta)$ satisfying $t=\Delta+o(\Delta)$ as $\Delta\to\infty$ such that the following holds. For any graph $G$ of maximum degree $\Delta$, and any $t$-thick partition $\mathcal{U}$ of maximum local degree $s$, there is guaranteed to be an independent transversal of $(G,\mathcal{U})$. Note this generalises the aforementioned result of Reed and Sudakov, as that could be cast in terms of independent transversals under the condition of maximum local degree $1$. As before, what is the best choice of $t=t(\Delta)$ that guarantees connectedness in the space of independent transversals under this setup?
\end{itemize}
For each of these, \autoref{thm: main} shows that we can choose $t\le2\Delta$, but is there some $1<c<2$ for which we can always choose $t \le c\Delta$ for all sufficiently large $\Delta$? Could it be any such $c$? And similarly, for each, one might separately investigate bounds on the threshold for rapid mixing (should one exist).

To cap things off, let's remind the reader of our first sentence in the paper. \autoref{thm: main} and the methods for proving it suggest the following supersaturation-type strengthening of Haxell's theorem.

\begin{conjecture}
Let $G$ be a graph with maximum degree $\Delta$ and let $\mathcal{U}$ be a $(2\Delta)$-thick partition of $V(G)$. The number of independent transversals of $(G,\mathcal{U})$ is minimised when $G$ is a disjoint union of $|\mathcal{U}|$ copies of $K_{\Delta,\Delta}$. 
\end{conjecture}

More specifically, it could be that the hypothesis of \autoref{lem: disagreement_all_classes} characterises the extremal structure, and so it would suffice to show that the number of independent transversals must be at least $2 \Delta^{|\mathcal{U}|}$.

\paragraph{Acknowledgements.} 
This project began during a visit by the authors in the summer of 2023 to the Department of Mathematics, Hokkaido University. We are grateful to Akira Sakai for hosting that visit.
PB was supported by the Dutch Research Council (NWO) grant OCENW.M20.009.
RJK was partially supported by the Gravitation Programme NETWORKS (024.002.003) of NWO.

\paragraph{Open access statement.} For the purpose of open access,
a CC BY public copyright license is applied
to any Author Accepted Manuscript (AAM)
arising from this submission.

\bibliographystyle{alphaurl}
\bibliography{biblio}

\end{document}